\pdfoutput=1
\documentclass[11pt]{amsart}

\usepackage{amsmath,amsthm, amssymb, amsfonts}
\usepackage[nohug]{diagrams}\diagramstyle[labelstyle=\scriptstyle]
\usepackage[mathscr]{eucal}
\usepackage{graphicx}
\usepackage[dvipsnames,usenames]{color}
\usepackage{subfigure}

\usepackage[colorlinks=true, urlcolor=NavyBlue, linkcolor=NavyBlue, citecolor=NavyBlue, pdftitle={Legendrian contact homology and nondestabilizability}, pdfauthor={Clayton Shonkwiler, David Shea Vela-Vick}, pdfsubject={Legendrian invariants}, pdfkeywords={Contact homology, Legendrian knots, 57R17; 57M25; 53D12; 53D40}]{hyperref}

\newtheorem{theorem}{Theorem}

\newtheorem{lemma}{Lemma}[section]
\newtheorem{proposition}[lemma]{Proposition}
\newtheorem{conjecture}[lemma]{Conjecture}

\theoremstyle{definition}

\theoremstyle{definition}
\newtheorem{remark}[lemma]{Remark}

\newcommand{\thmref}[1]{Theorem~\ref{#1}}

\newcommand{\mr}[1]{\mathrm{#1}}
\newcommand{\mscr}[1]{\mathscr{#1}}

\begin{document}

\thispagestyle{empty}

\title[Legendrian Contact Homology and Nondestabilizability]{Legendrian Contact Homology and Nondestabilizability}
\author{Clayton Shonkwiler}
\address{Department of Mathematics \\ Haverford College}
\email{cshonkwi@haverford.edu}
\urladdr{\href{http://www.haverford.edu/math/cshonkwi}{http://www.haverford.edu/math/cshonkwi}}

\author{David Shea Vela-Vick}
\thanks{DSV was partially supported by an NSF Postdoctoral Research Fellowship}
\address{Department of Mathematics \\ Columbia University}
\email{shea@math.columbia.edu}
\urladdr{\href{http://www.math.columbia.edu/~shea}{http://www.math.columbia.edu/\~{}shea}}

\date{\today}
\maketitle

\begin{abstract}
We provide the first example of a Legendrian knot with nonvanishing contact homology whose Thurston--Bennequin invariant is not maximal.
\end{abstract}

\section{Introduction} 
\label{sec:introduction}

Since it was proposed by Etnyre \cite{Et1} and first implemented by Etnyre and Honda \cite{EtHon1}, the most common strategy for classifying Legendrian knots in a given knot type $K$ has been to approach the problem in two steps. First, find all Legendrian representatives of $K$ with maximal Thurston--Bennequin invariant, then attempt to show that all other Legendrian representatives of $K$ can be destabilized to one of these maximal examples.

This method has proven quite effective, but, as observed by Etnyre and Honda \cite{EtHon2}, not all nondestabilizable Legendrian knots have maximal Thurston--Bennequin invariant. Thus, one needs a means for determining which Legendrian knots are nondestabilizable.

A candidate for identifying nondestabilizable Legendrian knots is Legendrian contact homology, which has been one of the most powerful nonclassical invariants of Legendrian knots since it was defined by Chekanov \cite{Ch} and Eliashberg \cite{El}. This invariant, which takes the form of a differential graded algebra $(\mscr{A}, \partial)$ and is a specialized variant of symplectic field theory \cite{ElGH}, vanishes for stabilized Legendrian knots and is nonvanishing for every nondestabilizable Legendrian knot for which it has been computed. All such examples to date have had maximal Thurston--Bennequin invariant, but in \thmref{thm:m10161} we show that the Legendrian contact homology is nonvanishing for a certain nondestabilizable Legendrian knot with nonmaximal Thurston--Bennequin invariant.

We do this by showing that a related invariant, the characteristic algebra, is nontrivial. The \emph{characteristic algebra} was defined by Ng \cite{Ng2} as $C(L) := \mscr{A}_F/\langle \mr{Im}\, \partial \rangle$ and is an invariant of the Legendrian knot $L$ up to tame isomorphism. Here $F$ is a front diagram for $L$, $\mscr{A}_F$ is the free, noncommutative, unital $\mathbb{Z}/2$-algebra generated by the crossings and right cusps of $F$, and $\langle \mr{Im}\, \partial \rangle \subset \mscr{A}_F$ is the two-sided ideal generated by the image of the contact homology differential.

Ng conjectured that the characteristic algebra of a nondestabilizable Legendrian knot is nonvanishing \cite[Conjecture 6.4.1]{Ng2}, which would imply that the Legendrian contact homology for such knots is also nonvanishing (see Proposition~\ref{pro:ca_to_ch_nonvanish}).  We give some evidence for Ng's conjecture by providing the first example of a Legendrian knot with nonvanishing characteristic algebra which does not have maximal Thurston--Bennequin invariant.

\begin{theorem}\label{thm:m10161}
The contact homology and characteristic algebra of Chongchitmate and Ng's nondestabilizable Legendrian $m(10_{161})$ are nonvanishing.
\end{theorem}

\begin{remark}
	A similar argument to that given in the proof of Theorem~\ref{thm:m10161} shows that the contact homology and characteristic algebra of Chongchitmate and Ng's nondestabilizable Legendrian $m(10_{145})$ are also nonvanishing.
\end{remark}

\begin{remark}
There is a lift of the contact homology and characteristic algebra to $\mathbb{Z}[t,t^{-1}]$.  Nonvanishing over $\mathbb{Z}/2$ implies nonvanishing in the more general $\mathbb{Z}[t,t^{-1}]$ setting.
\end{remark}

The general situation is still far from clear, however, as we also provide some evidence against Ng's conjecture.  Chongchitmate and Ng exhibited a Legendrian $m(10_{139})$ which does not have maximal Thurston--Bennequin invariant and which they conjectured, based on computational evidence, is nondestabilizable and sits atop its own peak in the $tb$--$r$ mountain range.  In Section~\ref{sec:char_alg_m10_139} we prove:

\begin{proposition}\label{pro:vanish}
	The contact homology and characteristic algebra of Chongchitmate and Ng's Legendrian $m(10_{139})$ vanish identically over $\mathbb{Z}[t,t^{-1}]$. 
\end{proposition}

Assuming this knot is actually nondestabilizable, this would provide the first example of a nondestabilizable Legendrian knot with vanishing characteristic algebra or contact homology.  This suggests the following:

\begin{conjecture}\label{conj:vanishing}
	There exist nondestabilizable Legendrian knots with vanishing contact homology.
\end{conjecture}

For background information on Legendrian knots and Legendrian contact homology, we refer the reader to Etnyre's survey \cite{Et2}.

\section*{Acknowledgements} 
\label{sec:acknowledgements}

We would like to thank John Etnyre for suggesting that we explore the relationship between the Thurston--Bennequin invariant and Legendrian contact homology.  Thanks also to Lenny Ng and Wutichai Chongchitmate for their work creating the Legendrian knot atlas \cite{Ng3}, Dylan Thurston for his helpful suggestion, and David Fithian for his time-saving \emph{Mathematica} program.

\section{The $m(10_{161})$} 
\label{sec:m10_161}

As mentioned in the introduction, Etnyre and Honda \cite{EtHon2} presented the first example of a nondestabilizable Legendrian knot whose Thurston--Bennequin invariant is nonmaximal for its knot type.  This example is a Legendrian $(2,3)$-cable of the $(2,3)$-torus knot.  

Recently, Chongchitmate and Ng produced a conjectural atlas \cite{Ng3} for low-crossing Legendrian knots.  Included in this atlas are several new examples of nondestabilizable Legendrian knots whose Thurston--Bennequin invariants are not maximal.  In particular, Chongchitmate and Ng give examples of nondestabilizable Legendrian $m(10_{161})$ and $m(10_{145})$ whose Thurston--Bennequin invariants are nonmaximal ($m$ here stands for ``mirror'').

For the purposes of computing the contact homology differential 
for a Legendrian knot, it is useful to have it presented as the plat closure of a positive braid.  Using Chongchitmate and Ng's original presentation, it is not difficult to derive the plat diagram for the $m(10_{161})$ appearing in Figure~\ref{fig:Ng}.

\begin{figure}[htbp]
	\centering
	\includegraphics[scale=0.7]{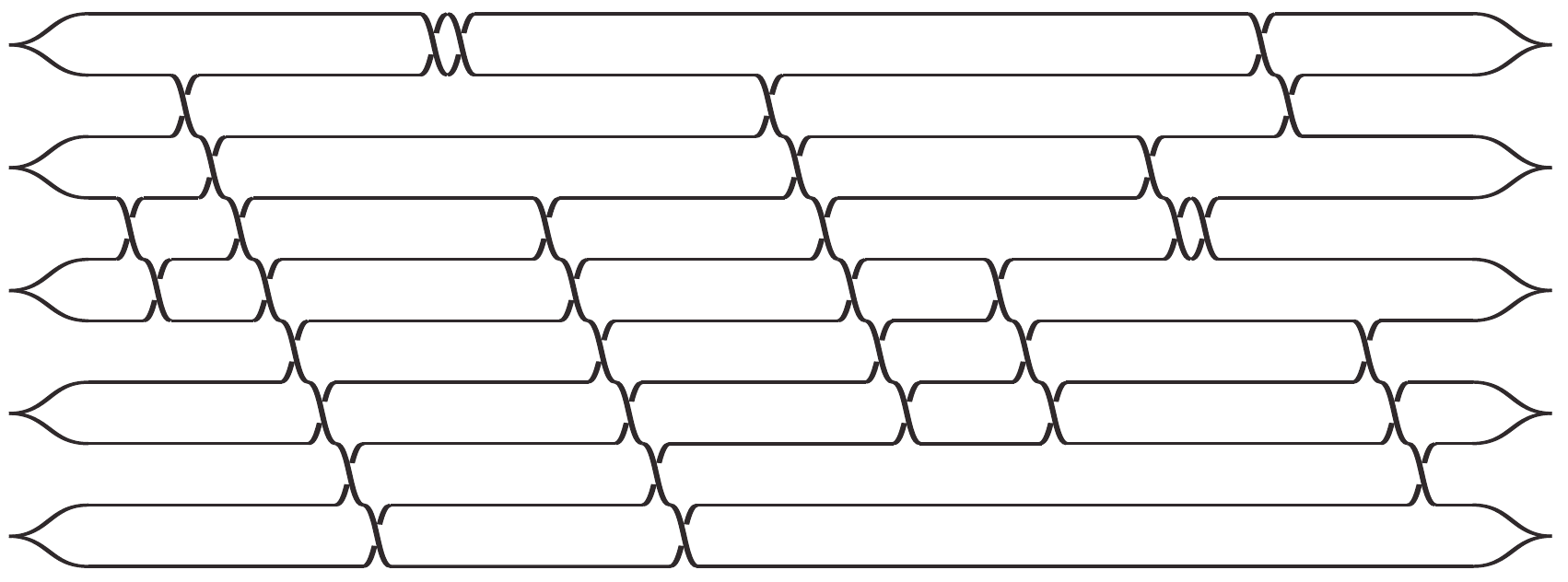}
	\caption{Chongchitmate and Ng's nondestabilizable $m(10_{161})$}
	\label{fig:Ng}
\end{figure}

The braid word defining the plat diagram in Figure~\ref{fig:Ng}
 is:\\

\noindent $4, 5, 2, 3, 4, 5, 6, 7, 8, 9, 1, 1, 4, 5, 6, 7, 8, 9, 2, 3, 4, 5, 6, 7, 5, 6, 7, 3, 4, 4, 1, 2, 6, 7, 8$\\

In Figure~\ref{fig:Ng} there are a total of 35 crossings and 5 right cusps.  The crossings are labeled $x_1$ through $x_{35}$ from left to right and the right cusps are labeled $x_{36}$ through $x_{40}$ from top to bottom. Therefore, for this front diagram for the $m(10_{161})$, $\mscr{A}_{m(10_{161})}$ is equal to $\mathbb{Z}/2 \langle x_{1}, \dots, x_{40} \rangle$, the free unital $\mathbb{Z}/2$-algebra of rank 40 generated by $x_{1}, \dots, x_{40}$.  The full boundary map is given in Appendix \ref{sec:the_differential}.


\section{The Proof of Theorem~\ref{thm:m10161}} 
\label{sec:proof_of_thm_1}


We begin with a straightforward observation relating (non)vanishing properties of the characteristic algebra to contact homology. 

\begin{proposition}\label{pro:ca_to_ch_nonvanish}
	Let $L$ be a Legendrian knot in the standard contact 3-sphere.  If the characteristic algebra of $L$ is nontrivial, then so is its contact homology.
\end{proposition}

\begin{proof}
	Suppose that the contact homology 
	\[
		\mathrm{CH}(L) = \frac{\ker(\partial)}{\hphantom{_{\ker}\!}\langle \mathrm{Im}\, \partial \rangle_{\ker}}
	\]
	of $L$ is trivial (here $\langle \mathrm{Im} \, \partial \rangle_{\ker}$ denotes the two-sided ideal generated by the image of the boundary map inside $\ker(\partial)$).  Then, since $\partial(1) = 0$, it must be the case that the unit element $1$ is contained in $\langle \mathrm{Im} \, \partial \rangle_{\ker}$.  However, since $\ker(\partial)$ is a subalgebra of $\mscr{A}_L$, this implies that $1$ must also be contained in the two-sided ideal $\langle \mathrm{Im} \, \partial \rangle$ generated by the image of the boundary map inside the full algebra $\mscr{A}_L$.  Therefore, the characteristic algebra of $L$ also vanishes, completing the proof of Proposition~\ref{pro:ca_to_ch_nonvanish}.
\end{proof}

By Proposition~\ref{pro:ca_to_ch_nonvanish}, \thmref{thm:m10161} will follow if we can show that the characteristic algebra of the Legendrian $m(10_{161})$ depicted in Figure~\ref{fig:Ng} is nontrivial.


The characteristic algebra $C(m(10_{161})) = \mscr{A}_{m(10_{161})} / \langle \mr{Im} \, \partial \rangle$ is 
\[
	C(m(10_{161}))= \mathbb{Z}/2 \langle x_1, \ldots , x_{40} \rangle / \langle \partial x_1, \ldots , \partial x_{40} \rangle.
\]
From the differential we have that
\begin{align*}
	\partial x_2 = x_1, & \quad \partial x_6 = x_3, \quad \partial x_5 = x_3 x_2 + x_4, \\
	\partial x_8 = x_7, & \quad \partial x_{10} = x_9, \quad \partial x_{15} = x_{14}, \\
	\partial x_{17} = x_{16}, & \quad \text{and} \quad \partial x_{26} = x_{25},
\end{align*}
so, in $C(m(10_{161}))$, 
\begin{equation}\label{eqn:original_zeros}
	x_1 = x_3 = x_4 = x_7 = x_9 = x_{14} = x_{16} = x_{25} = 0.
\end{equation}

To show that $C(m(10_{161})) \neq 0$ we will actually show that a quotient, $\overline{C} = C(m(10_{161})) / \mscr{I}$, is nontrivial.

Define $\mscr{I}$ as the two-sided ideal generated by the elements 
\begin{multline*}
	x_5, x_6, x_8, x_{10}, x_{15}, x_{17}, x_{18}, x_{19}, x_{20}, x_{21}, x_{22}, x_{23}, x_{24}, x_{26}, x_{31}, x_{32}, x_{35}, x_{36}, \\ x_{37}, x_{38}, x_{39},x_{40}, x_{30}+1, x_{34}+1 , x_{27}x_2 + 1, x_{11}x_{2}, x_{28} + x_{2}, x_{11} + x_{33}
\end{multline*}
and let
\[
	\overline{C} := C(m(10_{161}))/\mscr{I}.
\]

Using \eqref{eqn:original_zeros} and the relations of $\mscr{I}$, the defining relations of $C(m(10_{161}))$ (i.e. the boundary maps in Appendix~\ref{sec:the_differential}) can be simplified as
\begin{align}
\nonumber	x_2x_{13} + x_{12}x_{11} & = 1 \\
\nonumber	x_{11}x_{12} + x_{27}x_{12} & = 0\\
\nonumber	x_{13}x_{2} & =  1\\
\label{eqn:simplified_relations}	x_{11}(x_{29} + 1) & =  1\\
\nonumber	(x_{29} + 1)x_{11} + x_{2}x_{27} & =  1\\
\nonumber	x_{27}x_{12} & =  1 \\
\nonumber	x_{27}x_2 & = 1.
\end{align}
Therefore, $\overline{C}$ is isomorphic to $\mathbb{Z}/2\langle x_2, x_{11}, x_{12}, x_{13}, x_{27}, x_{29} \rangle$ modulo the relations in \eqref{eqn:simplified_relations}.

\begin{lemma}\label{lem:Aiso}
	The algebra $\overline{C}$ is isomorphic to the algebra
	\[
		\mathbb{Z}/2 \langle a,b,c,d \rangle / \langle ac+db=1, ba=0, bd=1, ca=1, cd=0 \rangle.
	\]
\end{lemma}

\begin{proof}
Define the map
\begin{align*}
x_{12} & \mapsto a \\
x_{13} & \mapsto b \\
x_{27} & \mapsto c \\
x_{29} + 1 & \mapsto d \\
x_{2} & \mapsto e \\
x_{11} & \mapsto f.
\end{align*}

Under this map, the relations in \eqref{eqn:simplified_relations} become
\begin{align}
\label{eqn:ebaf}	eb + af & = 1 \\
\label{eqn:faca}	fa + ca & = 0\\
\label{eqn:be}	be & =  1\\
\label{eqn:fd}	fd & =  1\\
\label{eqn:dfec}	df + ec & =  1\\
\label{eqn:ca}	ca & =  1 \\
\label{eqn:ce}	ce & = 1,
\end{align}
so $\overline{C}$ is isomorphic to $\mathbb{Z}/2\langle a,b,c,d,e,f\rangle$ modulo these relations. 

Note that, by adding \eqref{eqn:faca} to \eqref{eqn:ca}, the above relations imply
\begin{equation}
	\label{eqn:fa} fa = 1.
\end{equation}

Now, we claim that the relations in \eqref{eqn:ebaf}--\eqref{eqn:ce} are equivalent to the relations
\begin{align}
	\label{eqn:ca2}	ca & =  1 \\
	\label{eqn:bcf} b+c+f & = 0\\
	\label{eqn:ba} ba & = 0\\
	\label{eqn:ade} a+d+e & = 0\\
	\label{eqn:cd} cd & = 0\\
	\label{eqn:bd} bd & = 1\\
	\label{eqn:acdb} ac + db & = 1.
\end{align}

The relations \eqref{eqn:ebaf}--\eqref{eqn:ce} imply the relations \eqref{eqn:ca2}--\eqref{eqn:acdb} as follows:
\begin{itemize}
	\item The relation \eqref{eqn:ca2} already appears as \eqref{eqn:ca}.
	\item Multiply \eqref{eqn:ebaf} on the left by $c$ and simplify using \eqref{eqn:ca} and \eqref{eqn:ce} to get \eqref{eqn:bcf}.
	\item Multiply \eqref{eqn:bcf} on the right by $a$ and simplify using \eqref{eqn:fa} and \eqref{eqn:ca} to get \eqref{eqn:ba}.
	\item Multiply \eqref{eqn:dfec} on the right by $a$ and simplify using \eqref{eqn:fa} and \eqref{eqn:ca} to get \eqref{eqn:ade}.
	\item Multiply \eqref{eqn:ade} on the left by $c$ and simplify using \eqref{eqn:ca} and \eqref{eqn:ce} to get \eqref{eqn:cd}.
	\item Multiply \eqref{eqn:bcf} on the right by $d$ and simplify using \eqref{eqn:fd} and \eqref{eqn:cd} to get \eqref{eqn:bd}.
	\item Finally, multiply \eqref{eqn:bcf} on the left by $a$, multiply \eqref{eqn:ade} on the right by $b$, add the results and simplify using \eqref{eqn:ebaf} to get \eqref{eqn:acdb}.
\end{itemize}

On the other hand, we can derive \eqref{eqn:ebaf}--\eqref{eqn:ce} from \eqref{eqn:ca2}--\eqref{eqn:acdb} as follows:

\begin{itemize}
	\item Multiply \eqref{eqn:bcf} on the left by $a$, add to \eqref{eqn:acdb} and simplify using \eqref{eqn:ade} to get \eqref{eqn:ebaf}.
	\item Multiply \eqref{eqn:bcf} on the right by $a$ and simplify using \eqref{eqn:ca} and \eqref{eqn:ba} to get \eqref{eqn:faca}.
	\item Multiply \eqref{eqn:ade} on the left by $b$ and simplify using \eqref{eqn:ba} and \eqref{eqn:bd} to get \eqref{eqn:be}.
	\item Multiply \eqref{eqn:bcf} on the right by $d$ and simplify using \eqref{eqn:cd} and \eqref{eqn:bd} to get \eqref{eqn:fd}.
	\item Multiply \eqref{eqn:bcf} on the left by $d$, add to \eqref{eqn:acdb} and simplify using \eqref{eqn:ade} to get \eqref{eqn:dfec}.
	\item The relation \eqref{eqn:ca} appears as \eqref{eqn:ca2}.
	\item Multiply \eqref{eqn:ade} on the left by $c$ and simplify using \eqref{eqn:ca} and \eqref{eqn:cd} to get \eqref{eqn:ce}.
\end{itemize}

Therefore, since the two collections of relations \eqref{eqn:ebaf}--\eqref{eqn:ce} and \eqref{eqn:ca2}--\eqref{eqn:acdb} are equivalent, we see that
\begin{multline*}
	\overline{C} \simeq \mathbb{Z}/2 \langle a, b,c,d,e,f\rangle/\langle ca = 1, b+c+f=0, ba = 0, \\
	a+d+e = 0, cd=0, bd=1, ac+db=1\rangle.
\end{multline*}
Since $e = a+d$ and $f = b+c$, we can re-write $\overline{C}$ as 
\[
	\overline{C} \simeq \mathbb{Z}/2 \langle a,b,c,d \rangle / \langle ac+db=1, ba=0, bd=1, ca=1, cd=0\rangle,
\]
completing the proof of the lemma.
\end{proof}

The goal now is to show that $\overline{C}$ is nontrivial, which will imply that $C(m(10_{161}))$ is nontrivial as well. 

\begin{lemma}\label{lem:cbar_nontrivial}
	The algebra 
	\[
		\overline{C} = \mathbb{Z}/2 \langle a,b,c,d \rangle / \langle ac+db=1, ba=0, bd=1, ca=1, cd=0\rangle
	\]
	is nontrivial.
\end{lemma}

\begin{proof}
	To prove this, we define an action of $\overline{C}$ on $\mscr{H}$, where $\mscr{H}$ is a countably infinite-dimensional vector space over $\mathbb{Z}/2$.  Provided we can show this action is nontrivial, this will imply that $\overline{C}$ is nontrivial.

As with any infinite-dimensional vector space, $\mscr{H}$ can be written as
\[
	\mscr{H} = \mscr{H}_1 \oplus \mscr{H}_2,
\]
where $\mscr{H}_1 \simeq \mscr{H}_2 \simeq \mscr{H}$ as $\mathbb{Z}/2$-vector spaces, so any map $\mscr{H} \to \mscr{H}_1 \oplus \mscr{H}_2$ or $\mscr{H}_1 \oplus \mscr{H}_2 \to \mscr{H}$ defines an endomorphism of $\mscr{H}$.

Fix identifications $\mscr{H} \cong \mscr{H}_1$ and $\mscr{H} \cong \mscr{H}_2$ (throughout what follows the symbol $\cong$ will refer to these fixed identifications).

Let $a, b,c,d$ act on $\mscr{H}$ as follows:
\begin{itemize}
	\item Define $a: \mscr{H} \to \mscr{H}_1 \oplus \mscr{H}_2$ by the diagram
	\[
		\begin{diagram}[height=.7em]
			& & \mathcal{H}_1 \\
			& \ruTo^{0} & \\
			\mathcal{H} & & \oplus \\
			& \rdTo_{\cong} & \\
			& & \mathcal{H}_2. \\
		\end{diagram}		
	\]
	\item Define $b: \mscr{H}_1 \oplus \mscr{H}_2 \to \mscr{H}$ by the diagram 
	\[
		\begin{diagram}[height=.7em]
			\mathcal{H}_1 & & \\
			& \rdTo^{\cong} & \\
			\oplus & & \mathcal{H}.\\
			& \ruTo_{0} & \\
			\mathcal{H}_2 & &  \\
		\end{diagram}
	\]
	\item Define $c: \mscr{H}_1 \oplus \mscr{H}_2 \to \mscr{H}$ by the diagram
	\[
		\begin{diagram}[height=.7em]
			\mathcal{H}_1 & & \\
			& \rdTo^{0} & \\
			\oplus & & \mathcal{H}.\\
			& \ruTo_{\cong} & \\
			\mathcal{H}_2 & &  \\
			\end{diagram}		
	\]
	\item Define $d: \mscr{H} \to \mscr{H}_1 \oplus \mscr{H}_2$ by the diagram
	\[
		\begin{diagram}[height=.7em]
			& & \mathcal{H}_1 \\
			& \ruTo^{\cong} & \\
			\mathcal{H} & & \oplus \\
			& \rdTo_{0} & \\
			& & \mathcal{H}_2. \\
		\end{diagram}
	\]
\end{itemize}

Extending by linearity, the defining relations of $\overline{C}$ are preserved by this action, so the above induces a well-defined action of $\overline{C}$ on $\mscr{H}$ (alternatively, a representation of $\overline{C}$ into $\text{End}(\mscr{H})$).  Since the actions of $a$, $b$, $c$, and $d$ are clearly nontrivial, this is a nontrivial action, completing the proof of the lemma.
\end{proof}

Since $\overline{C}$ is a quotient of $C(m(10_{161}))$, Lemma~\ref{lem:cbar_nontrivial} implies that $C(m(10_{161}))$ is nontrivial, completing the proof of Theorem~\ref{thm:m10161}.


\section{The Contact Homology of the $m(10_{139})$} 
\label{sec:char_alg_m10_139}

Our goal in this section is to prove Proposition~\ref{pro:vanish} by showing that $1$ is in the image of the differential of Chongchitmate and Ng's conjecturally nondestabilizable  $m(10_{139})$.  This Legendrian $m(10_{139})$ is one of two examples given by Chongchitmate and Ng with nonmaximal Thurston--Bennequin invariants which computations suggest sit atop their own peaks in the $tb$--$r$ mountain range.  That the other---a Legendrian $m(12n_{242})$---also has vanishing contact homology and characteristic algebra follows from a similar argument to the one given below.  

\begin{figure}[htbp]
	\centering
		\includegraphics[scale=0.65]{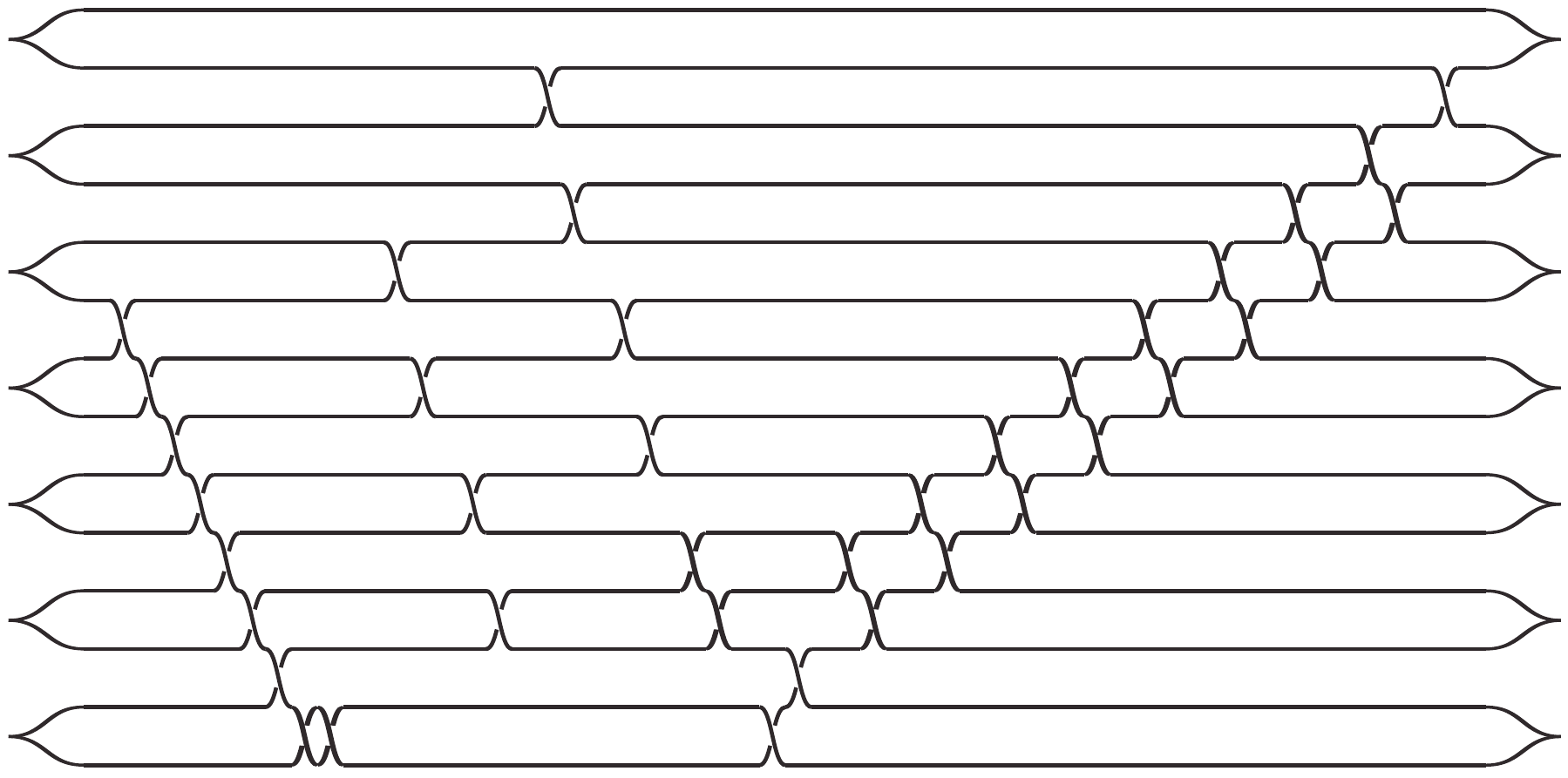}
	\caption{Chongchitmate and Ng's conjecturally nondestabilizable $m(10_{139})$}
	\label{fig:10_139}
\end{figure}

The plat diagram for the $m(10_{139})$ given in Figure~\ref{fig:10_139} is obtained from Chongchitmate and Ng's presentation.

The braid word for the plat diagram in Figure~\ref{fig:10_139} is:

\vspace{.1in}
\begin{center}
	\noindent$6, 7, 8, 9, 10, 11, 12, 13, 13, 5, 7, 9, 11, 2, 4, 6, 8, 10, 11,$\\ $13, 12, 10, 11, 9, 10, 8, 9, 7, 8, 6, 7, 5, 6, 4, 5, 3, 4, 2$
\end{center}
\vspace{.1in}

In order to prove that the contact homology and characteristic algebra of the $m(10_{139})$ are trivial, it suffices to construct an element   $a \in A$  such that $\partial a = 1$. From the presentation of the differential given in Appendix~\ref{sec:the_differential_m10_139}, we see that
\begin{multline*}
	1 = \partial \left(\vphantom{\frac{x}{y}}\!(x_2+x_{10})\!\left(\vphantom{\sqrt{x^2}}\!\!\left(\vphantom{x^{12}}(x_{41}x_{11} + x_{14}x_{42})x_{15} + x_{41}-x_{44}\right)x_{22} + x_{24}\right) \right.\\
	\left.\vphantom{\frac{x}{y}}+ (x_4+x_{16})(x_{15}x_{22}+x_{19}) + x_{6} + x_{43}\!\right).
\end{multline*}

Therefore, the contact homology and characteristic algebra of Chongchitmate and Ng's $m(10_{139})$ both vanish over $\mathbb{Z}[ t , t^{-1} ]$.


\bigskip

\bigskip


\bigskip

\bigskip

\appendix

\section{The Differential over $\mathbb{Z}/2$ for the $m(10_{161})$} 
\label{sec:the_differential}

\begin{itemize}

		\item $\partial x_{1} =    0$

		\item $\partial x_{2} =    x_{1}$

		\item $\partial x_{3} =    0$

		\item $\partial x_{4} =    x_{3}x_{1}$

		\item $\partial x_{5} =    x_{3}x_{2} + x_{4}$

		\item $\partial x_{6} =    x_{3}$

		\item $\partial x_{7} =    0$

		\item $\partial x_{8} =    x_{7}$

		\item $\partial x_{9} =    0$

		\item $\partial x_{10} =   x_{9}$

		\item $\partial x_{11} =   x_{3}$

		\item $\partial x_{12} =   0$

		\item $\partial x_{13} =   0$

		\item $\partial x_{14} =   0$

		\item $\partial x_{15} =   x_{14}$

		\item $\partial x_{16} =   0$

		\item $\partial x_{17} =   x_{16}$

		\item $\partial x_{18} =   0$

		\item $\partial x_{19} =   x_{1} + x_{12}x_{4} + x_{12}x_{11}x_{1}$

		\item $\partial x_{20} =   x_{2}x_{13} + x_{12}x_{5}x_{13} + x_{12}x_{11}x_{2}x_{13} + 1 + x_{12}x_{6} + x_{12}x_{11} + x_{19}x_{13}$

		\item $\partial x_{21} =   x_{2}x_{14} + x_{12}x_{5}x_{14} + x_{12}x_{11}x_{2}x_{14} + x_{12}x_{7} + x_{19}x_{14}$

		\item $\partial x_{22} =   x_{2}x_{15} + x_{12}x_{5}x_{15} + x_{12}x_{11}x_{2}x_{15} + x_{12}x_{8} + x_{19}x_{15} + x_{21}$

		\item $\partial x_{23} =   x_{2}x_{16} + x_{12}x_{5}x_{16} + x_{12}x_{11}x_{2}x_{16} + x_{12}x_{9} + x_{19}x_{16}$

		\item $\partial x_{24} =   x_{2}x_{17} + x_{12}x_{5}x_{17} + x_{12}x_{11}x_{2}x_{17} + x_{12}x_{10} + x_{19}x_{17} + x_{23}$

		\item $\partial x_{25} =   0$

		\item $\partial x_{26} =   x_{25}$

		\item $\partial x_{27} =   0$

		\item $\partial x_{28} =   0$

		\item $\partial x_{29} =   x_{28}x_{25}$

		\item $\partial x_{30} =   0$

		\item $\partial x_{31} =   x_{4} + x_{11}x_{1}$

		\item $\partial x_{32} =   x_{5}x_{13}x_{28} + x_{11}x_{2}x_{13}x_{28} + x_{6}x_{28} + x_{11}x_{28} + x_{5}x_{14} + x_{11}x_{2}x_{14} + x_{7} + x_{31}x_{13}x_{28} + x_{31}x_{14}$

		\item $\partial x_{33} =   0$

		\item $\partial x_{34} =   0$

		\item $\partial x_{35} =   x_{33}x_{2}x_{18} + x_{33}x_{12}x_{5}x_{18} + x_{33}x_{12}x_{11}x_{2}x_{18} + x_{33}x_{12} + x_{33}x_{19}x_{18} + x_{34}x_{27}x_{2}x_{18} + x_{34}x_{27}x_{12}x_{5}x_{18} + x_{34}x_{27}x_{12}x_{11}x_{2}x_{18} + x_{34}x_{27}x_{12} + x_{34}x_{27}x_{19}x_{18}$

		\item $\partial x_{36} =   x_{13}x_{28} + x_{14} + 1$

		\item $\partial x_{37} =   x_{5}x_{13}x_{29}x_{30} + x_{11}x_{2}x_{13}x_{29}x_{30} + x_{6}x_{29}x_{30} + x_{11}x_{29}x_{30} + x_{5}x_{15}x_{25}x_{30} + x_{11}x_{2}x_{15}x_{25}x_{30} + x_{8}x_{25}x_{30} + x_{5}x_{16}x_{30} + x_{11}x_{2}x_{16}x_{30} + x_{9}x_{30} + x_{5}x_{13} + x_{11}x_{2}x_{13} + x_{6} + x_{11} + x_{31}x_{13}x_{29}x_{30} + x_{31}x_{15}x_{25}x_{30} + x_{31}x_{16}x_{30} + x_{31}x_{13} + x_{32}x_{25}x_{30} + 1$

		\item $\partial x_{38} =   x_{33} + x_{30}x_{28}x_{26}x_{33} + x_{30}x_{29}x_{33} + x_{30}x_{28}x_{27} + 1$

		\item $\partial x_{39} =   x_{27}x_{2}x_{18} + x_{27}x_{12}x_{5}x_{18} + x_{27}x_{12}x_{11}x_{2}x_{18} + x_{27}x_{12} + x_{27}x_{19}x_{18} + 1$

		\item $\partial x_{40} =   x_{33}x_{2} + x_{33}x_{12}x_{5} + x_{33}x_{12}x_{11}x_{2} + x_{33}x_{19} + x_{34}x_{27}x_{2} + x_{34}x_{27}x_{12}x_{5} + x_{34}x_{27}x_{12}x_{11}x_{2} + x_{34}x_{27}x_{19} + 1$
\end{itemize}


\section{The Differential over $\mathbb{Z}[t,t^{-1}]$ for the $m(10_{139})$} 
\label{sec:the_differential_m10_139}

\begin{itemize}

		\item $\partial x_{1} =    0$

		\item $\partial x_{2} =    -x_{1}$

		\item $\partial x_{3} =    0$

		\item $\partial x_{4} =    -x_{3}$

		\item $\partial x_{5} =    0$

		\item $\partial x_{6} =    -x_{5}$

		\item $\partial x_{7} =    0$

		\item $\partial x_{8} =    -x_{7}$

		\item $\partial x_{9} =    0$

		\item $\partial x_{10} =   x_{1}$

		\item $\partial x_{11} =   0$

		\item $\partial x_{12} =   0$

		\item $\partial x_{13} =   0$

		\item $\partial x_{14} =   0$

		\item $\partial x_{15} =   0$

		\item $\partial x_{16} =   x_2 x_{11} + x_{10} x_{11} + x_3$

		\item $\partial x_{17} =   x_{11} x_{12}$

		\item $\partial x_{18} =   x_{12} x_{13}$

		\item $\partial x_{19} =   x_{12}$

		\item $\partial x_{20} =   0$

		\item $\partial x_{21} =   x_{18} x_{9} x_{20} - x_{19} x_{13} x_{9} x_{20} + x_{18} - x_{19} x_{13}$

		\item $\partial x_{22} =   0$

		\item $\partial x_{23} =   x_{9} x_{20} + 1 + x_{22} x_{13} x_{9} x_{20} + x_{22} x_{13}$

		\item $\partial x_{24} =   x_{11} x_{18} x_{22} + x_{17} x_{13} x_{22} + x_{11} x_{19} +x_{17}$

		\item $\partial x_{25} =   x_{11} x_{18} x_{23} + x_{17} x_{13} x_{23} + x_{11} x_{21} - x_{24} x_{13} x_{9} x_{20} - x_{24} x_{13}$

		\item $\partial x_{26} =   x_{13} x_{22} + 1$

		\item $\partial x_{27} =   x_{13} x_{23} + x_{26} x_{13} x_{9} x_{20} + x_{26} x_{13}$

		\item $\partial x_{28} =   x_{2} x_{17} x_{26} + x_{10} x_{17} x_{26} + x_{4} x_{12} x_{26} + x_{5} x_{26} + x_{16} x_{12} x_{26} + x_{2} x_{24} + x_{10} x_{24} + x_{4} x_{18} x_{22} + x_{16} x_{18} x_{22} + x_{6} x_{13} x_{22} + x_{7} x_{22} + x_{4} x_{19} + x_{16} x_{19} + x_{6}$

		\item $\partial x_{29} =   x_{2} x_{17} x_{27} + x_{10} x_{17} x_{27} + x_{4} x_{12} x_{27} + x_{5} x_{27} + x_{16} x_{12} x_{27} + x_{2} x_{25} + x_{10} x_{25} + x_{4} x_{18} x_{23} + x_{16} x_{18} x_{23} + x_{6} x_{13} x_{23} + x_{7} x_{23} + x_{4} x_{21} + x_{16} x_{21} + x_{8} x_{9} x_{20} + x_{20} + x_{8} - x_{28} x_{13} x_{9} x_{20} - x_{28} x_{13}$

		\item $\partial x_{30} =   x_{12} x_{26} + x_{18} x_{22} + x_{19}$
		
		\item $\partial x_{31} =    x_{12} x_{27} + x_{18} x_{23} + x_{21} + x_{30} x_{13} x_{9} x_{20} + x_{30} x_{13}$

		\item $\partial x_{32} =    x_{15} x_{11} x_{30} + x_{15} x_{17} x_{26} + x_{15} x_{24}$

		\item $\partial x_{33} =    x_{15} x_{11} x_{31} + x_{15} x_{17} x_{27} + x_{15} x_{25} - x_{32} x_{13} x_{9} x_{20} - x_{32} x_{13}$

		\item $\partial x_{34} =    x_{11} x_{30} + x_{17} x_{26} + x_{24}$

		\item $\partial x_{35} =    x_{11} x_{31} + x_{17} x_{27} + x_{25} + x_{34} x_{13} x_{9} x_{20} + x_{34} x_{13}$

		\item $\partial x_{36} =    x_{14} x_{15} x_{34} + x_{14} x_{32}$

		\item $\partial x_{37} =    x_{14} x_{15} x_{35} + x_{14} x_{33} - x_{36} x_{13} x_{9} x_{20} - x_{36} x_{13}$

		\item $\partial x_{38} =    x_{15} x_{34} + x_{32}$

		\item $\partial x_{39} =    x_{14} x_{38} + x_{36} + 1$

		\item $\partial x_{40} =   x_{15} x_{35} + x_{33} + x_{38} x_{13} x_{9} x_{20} + x_{38} x_{13} + 1$

		\item $\partial x_{41} =   x_{14} x_{15} + 1$

		\item $\partial x_{42} =   x_{15} x_{11} + 1$

		\item $\partial x_{43} =   x_{2} x_{17} + x_{10} x_{17} + x_{4} x_{12} + x_{5} + x_{16} x_{12} + 1$

		\item $\partial x_{44} =   x_{11} x_{18} + x_{17} x_{13} + 1$

		\item $\partial x_{45} =   x_{18} x_{9} - x_{19} x_{13} x_{9} + t^{-1}$
\end{itemize}


\end{document}